\documentclass[final]{siamltex}
\usepackage{array}
\usepackage{amssymb}
\usepackage{bm}
\usepackage{mathtools}
\usepackage{amsmath}
\usepackage{comment}
\usepackage{amsfonts}
\usepackage{epstopdf}
\usepackage{algorithm}
\usepackage{textcomp}
\usepackage{hyperref}
\usepackage{tikz}
\usepackage{graphicx,color} %
\usepackage{subfigure}

\graphicspath{{img/all/}}

\ifpdf
  \DeclareGraphicsExtensions{.eps,.pdf,.png,.jpg}
\else
  \DeclareGraphicsExtensions{.eps}
\fi

\newcommand{\TODO}[1][]{{\color{red}TODO\ifthenelse{\equal{#1}{}}{}{: #1}}}

  \usepackage{todonotes}

\newcommand{\removed}[1] {\ifmmode{\color{red}\cancel{#1}}\else{\color{red}\sout{#1}}\fi}

\newcommand{\scurl}[1][]{\textup{curl}\ifthenelse{\equal{#1}{}}{}{_#1}}
\newcommand{\vcurl}[1][]{\textup{\textbf{curl}}\ifthenelse{\equal{#1}{}}{}{_#1}}

\newcommand{\sdiv}[1][]{\textup{div}\ifthenelse{\equal{#1}{}}{}{_#1}}
\newcommand{\vdiv}[1][]{\textup{\textbf{div}}\ifthenelse{\equal{#1}{}}{}{_#1}}

\newcommand{\RR}{\mathbb{R}}

\def\cT{\mathcal{T}}
\newcommand{\norm}[1]{\| #1 \|}

\begin{document}

\title{A Stabilized Trace FEM for Surface Cahn--Hilliard Equations: Analysis and Simulations}

\author{Deepika Garg\thanks{Department of Mathematics, University of Houston, 651 PGH, Houston, 77204, TX, USA, 
 (dgarg5@uh.edu, maolshanskiy@uh.edu ). } 
\and Maxim Olshanskii\footnotemark[1]
}



\maketitle
\begin{abstract}
	This paper addresses the analysis and numerical assessment of a computational method for solving the Cahn--Hilliard equation defined on a surface. The proposed approach combines the stabilized trace finite element method for spatial discretization with an implicit–explicit scheme for temporal discretization. The method belongs to a class of unfitted finite element methods that use a fixed background mesh and a level-set function for implicit surface representation. We establish the numerical stability of the discrete problem by showing a suitable energy dissipation law for it. We further derive optimal-order error estimates assuming simplicial background meshes and finite element spaces of order $m \geq 1$. The effectiveness of the method is demonstrated through numerical experiments on several two-dimensional closed surfaces, confirming the theoretical results and illustrating the robustness and convergence properties of the scheme.
\end{abstract}



\section{Introduction}

The Cahn--Hilliard equation \cite{cahn1958free} models phase separation in binary mixtures, where the solution variable represents concentration and a double-well potential characterizes the distinct phases. While originally developed for metallurgical applications such as spinodal decomposition in alloys \cite{cahn1961spinodal, miller1995spinodal}, the equation and its variants have found widespread use in diverse fields, including modeling of tumor growth  \cite{garcke2022viscoelastic, garcke2016cahn} and multicomponent plasma membranes \cite{membr1,zhiliakov2021experimental}.

Traditional numerical studies of the Cahn--Hilliard equation have predominantly focused on stationary Euclidean domains; see, e.g., \cite{elliott1989second, shen2010numerical}, among many other publications. More recently, attention has been directed to surface-bound formulations motivated by applications to fluid membrane thermodynamics. The surface Cahn--Hilliard equation poses particular numerical challenges due to the presence of fourth-order terms, nonlinearity, and the difficulties of discretizing tangential operators on complex geometries. While finite difference methods have been applied to specific surface configurations \cite{gera2017cahn,greer2006fourth, jeong2015microphase}, several finite element methods have been developed for surface Cahn--Hilliard equations, including fitted approaches on spheres and saddle surfaces \cite{du2011finite}, with extensions to more general surfaces \cite{garcke2016coupled, li2017unconditionally}. Further developments include coupled systems such as surface Cahn--Hilliard--Navier--Stokes equations \cite{nitschke2012finite}, bulk–surface coupled models \cite{barrett2017finite}, and evolving surface formulations \cite{elliott2025fully}. A common limitation of these methods is their reliance on surface-fitted meshes that conform exactly to the geometry.

In this paper, we consider a geometrically unfitted finite element method for the surface Cahn--Hilliard equation that has been introduced and applied to simulate phase separation in multicomponent lipid bilayers in \cite{membr1,wang2023fusogenicity,yushutin2019computational,zhiliakov2021experimental}. These earlier studies lacked theoretical analysis, and the present work aims to fill this gap by proving stability and deriving optimal-order error estimates for the method.

Our approach builds on the Trace Finite Element Method (TraceFEM) \cite{olshanskii2018trace}, which employs a sharp surface representation without requiring surface-fitted meshes or explicit parametrization. Unlike fitted methods that necessitate mesh generation conforming to the surface geometry, the unfitted approach uses a fixed background mesh that is independent of the surface position. The surface $\Gamma$ is represented implicitly as the zero level set of a smooth function $\phi$ defined on the background domain, allowing for greater flexibility in handling complex geometries.

The key features of the method include: (i) stabilization terms that ensure well-conditioned linear systems while controlling the extension of solutions from the surface to the volumetric neighborhood; (ii) energy-stable time discretization that preserves the gradient flow structure of the continuous problem; and (iii) optimal-order error estimates for arbitrary-order finite element spaces.

Numerical stability of the method follows from  a discrete energy dissipation law established here. We also derived an  optimal-order error estimate for simplicial background meshes using finite element spaces of order $m \geq 1$. The error estimate includes a quantification of the geometrical error. The effectiveness of the method is demonstrated through numerical experiments on various surfaces, confirming the theoretical convergence rates and demonstrating its capability for simulating phase separation dynamics on complex geometries.

The outline of this paper is as follows. In Section \ref{sec:model}, we introduce the model problem and its weak formulation. Section \ref{sec:preliminaries} presents the preliminaries of the trace finite element method. In Section \ref{sec:stability}, we establish a stability result. Section \ref{sec:error_analysis} provides error estimates for the fully discrete scheme in both space and time. Finally, Section \ref{sec:numerics} presents results of a few numerical experiments. We first verify the theoretical convergence rates using piecewise linear and quadratic finite elements, and then demonstrate the effectiveness of the method through phase separation examples on a sphere and on a more complex surface.

\section{Problem Definition}\label{sec:model}

Let $\Gamma$ be a closed, sufficiently smooth surface in $\mathbb{R}^3$ with outward unit normal $\mathbf{n}$.  
We consider the standard Sobolev space \( H^{r}(\Gamma) \) of order \( r \geq 0 \) for functions defined on \(\Gamma\), equipped with the norm \(\| \cdot \|_{H^{r}(\Gamma)}\). 
We consider the Cahn--Hilliard equation \cite{cahn1961spinodal, cahn1958free} posed on $\Gamma$, which models phase separation without phase transition in binary systems within thin material layers. 

The Cahn--Hilliard system for concentration $c$ and chemical potential $\mu$ reads:
\begin{align}
	\frac{\partial c}{\partial t} - \mbox{div}_{\Gamma}(M \nabla_{\Gamma} \mu) = 0& \quad \text{on } \Gamma \times (0, T], \label{eq_1} \\
	\mu = f'_0(c) - \epsilon^2 \Delta_{\Gamma}c& \quad \text{on } \Gamma\times (0, T], \label{eq_3}
\end{align}
where $M$ is the mobility coefficient, $\mu$ is the chemical potential, and $\epsilon > 0$ characterizes the thickness of a  transition layer between two phases. The system is supplemented with the initial condition $c(x,0) = c_0(x)$ on $\Gamma$.

We will assume the constant mobility $M = 1$ and consider the Landau-Lifshitz  double-well potential: 
\begin{align*}
	f_0(c) = \frac{1}{4}(c^2-1)^2, 
\end{align*}
which has equal minima at $c=\pm1$, so that $c=-1$ corresponds to phase 1 and  $c=1$ corresponds to phase 2.

The total free energy functional is given by:
\begin{align}
	E(c) = \int_{\Gamma} \left( \frac{1}{2} \epsilon^2 |\nabla_{\Gamma}c|^2 + f_0(c) \right) ds, \label{eq_6}
\end{align}
and satisfies the energy dissipation property 

\begin{align}\label{eq_7}
	\frac{d}{dt}E(c) < 0.   
\end{align}

A weak formulation of (\ref{eq_1})--(\ref{eq_3}) reads:\\
Find $ (c, \mu) \in \left(H^1(0,T;H^{-1}(\Gamma)) \cap L^2(0,T;H^1(\Gamma))\right) \times L^2(0,T;H^1(\Gamma))$ such that 
\begin{align} \label{eq_8}
	\int_{\Gamma}\frac{\partial c}{\partial t} v \,ds &= - \int_{\Gamma} M \nabla_{\Gamma} \mu \cdot \nabla_{\Gamma}v \,ds, \nonumber \\
	\int_{\Gamma}\mu q  \,ds &= \int_{\Gamma} f'_0(c) q\,ds + \epsilon^2\int_{\Gamma}  \nabla_{\Gamma} c \cdot \nabla_{\Gamma}q \,ds.
\end{align}
for all $(v, q) \in H^1(\Gamma) \times H^1(\Gamma)$ and a.e. $t\in(0,T)$. 


\section{Preliminaries}\label{sec:preliminaries}

For the space discretization of the surface Cahn--Hilliard problem described in the previous section, we apply the TraceFEM \cite{olshanskii2018trace, yushutin2019computational}. 

As typical for geometrically unfitted or immersed interface/boundary methods, the TraceFEM relies on triangulation of a bulk computational domain $\Omega$ (such that $\Gamma \subset \Omega$ holds) into shape regular tetrahedra that are blind to the position of $\Gamma$. The surface ifself is defined implicitly as the zero level set of a sufficiently smooth (at least Lipschitz continuous) function $\phi$, i.e., 
\[
\Gamma = \{x \in \Omega : \phi(x) = 0\},
\]
such that $|\nabla \phi| \ge c_0 > 0$ in $\mathcal{O}(\Gamma)$, a 3D neighborhood  of the surface.

Let $\{\mathcal{T}_h\}$ be a collection of all tetrahedra such that $\overline{\Omega}= \cup_{T \in \mathcal{T}_h} \overline{T}$. 
For the purpose of analysis we assume that the mesh $\mathcal{T}_h$ is  quasi-uniform with the characteristic mesh size $h$.
Consider a standard finite element space of continuous functions that are piecewise polynomials of degree $ \le m$. This bulk (volumetric) finite element space is denoted by $V_h^{\rm bulk}$:
\[
V_h^{\rm bulk} := \left\{w \in {C}(\Omega) : w \in \mathcal{P}_m(T) ~~ \forall ~ T \in \cT_{h}   \right\}.
\]

To define geometric quantities and for the purpose of numerical integration, we approximate $\Gamma$ with a discrete surface $\Gamma_h$, which is defined as the zero level set of a Lagrangian interpolant $\phi_h\in V_h^{\rm bulk}$ with $m=q$ for the level set function $\phi$ on the given mesh. The positive integer $q$ defines the order of the geometry approximation. The problem of efficient numerical integration over $\Gamma_h$ is addressed in \cite{GLR:2018}. 

The discrete normal vector and tangential projector are given by
\[
\mathbf{n}_h = \frac{\nabla \phi_h}{|\nabla \phi_h|},
\qquad
\mathbf{P}_h = I - \mathbf{n}_h \mathbf{n}_h^{\!T},
\]
and the surface gradient on~$\Gamma_h$ reads $\nabla_{\Gamma_h}v = \mathbf{P}_h\nabla v$ for $v\in V_h^{\rm bulk}$.
The geometry approximation satisfies
\begin{equation*}
	\|\mathbf{n}^e - \mathbf{n}_h\|_{L^\infty(\Gamma_h)} \le C h^q,
	\quad
	\mathrm{dist}(\Gamma_h, \Gamma) \le C h^{q+1},
\end{equation*}
where $\mathbf{n}^e$ is the closest point extension of $\mathbf{n}$ off $\Gamma$ as defined below.
Here and further the scalar product in \( L^{2}(\omega) \) is  denoted by \((\cdot,\cdot)_{\omega}\), where $\omega$ can be either a domain in $\RR^d$ or a surface of co-dimension one.  The corresponding norm is written as \(\| \cdot \|_{\omega}\).  
We will also use $C$ to denote a generic constant independent of discretization parameters. 

The subset of tetrahedra that have a nonzero intersection with $\Gamma_h$ is denoted by $\mathcal{T}^{\Gamma}_h$.  The domain formed by all tetrahedra in $\mathcal{T}^{\Gamma}_h$ is denoted by $\Omega^{\Gamma}_h$.

On $\mathcal{T}^{\Gamma}_h$ we use restrictions of functions from $ V_h^{\rm bulk}$. This narrow-band (volumetric) finite element space is denoted by $V_h$:
\[
V_h := \left\{w \in {C}(\Omega^{\Gamma}_h) : w \in \mathcal{P}_m(T) ~~ \forall ~ T \in \cT^{\Gamma}_{h}   \right\}.
\]
It will serve as a space for trial and test functions in the finite element method.

{For the analysis of finite element methods on surfaces, it is crucial to relate functions and norms between the exact surface $\Gamma$ and its discrete approximation $\Gamma_h$. This is facilitated by extension and lift operations. For a function $v : \Gamma \to \mathbb{R}$, its {normal extension} $v^e$ to a tubular neighborhood $\mathcal{O}(\Gamma)$ is defined by $v^e(\mathbf{x}) = v(\mathbf{p}(\mathbf{x}))$, where $\mathbf{p}(\mathbf{x})$ is the closest point projection onto $\Gamma$. Conversely, for a function $v_h : \Gamma_h \to \mathbb{R}$, its {lift} $v_h^\ell$ onto $\Gamma$ is defined by $v_h^\ell(\mathbf{p}(\mathbf{x})) = v_h(\mathbf{x})$ for all $\mathbf{x} \in \Gamma_h$.
	We have the following norm equivalence results:}
\begin{itemize}
	\item \textit{Uniform norm bound for extended functions}:  For a function $v \in W^{k,\infty}(\Gamma)$, its normal extension is stable in the sense that
	\begin{equation*}
		\|v^e\|_{W^{k,\infty}(\mathcal{O}(\Gamma))} \le C \|v\|_{W^{k,\infty}(\Gamma)}. \label{eq_30}
	\end{equation*}
	Here $\mathcal{O}(\Gamma)$ is a tubular neighborhood of th surface  where the closest point projection is well-defined.
	\item \textit{Volume-to-surface bound in narrow bands}: { For $\epsilon > 0$ sufficiently small, let $U_\epsilon(\Gamma) = \{ \mathbf{x} \in \mathbb{R}^3 : |\phi(\mathbf{x})| < \epsilon \} \subset \mathcal{O}(\Gamma)$ be a narrow band. Then for the normal extension $v^e$,}
	\begin{equation*}
		\|v^e\|_{H^k(U_\epsilon(\Gamma))} \le C \epsilon^{1/2} \|v\|_{H^k(\Gamma)}. \label{eq_31}
	\end{equation*}
	The factor $\epsilon^{1/2}$ accounts for the integration over the normal direction.
    \item \textit{Geometric identities.}
For $x\in\Gamma_h$, let $\mathbf{p}(x)\in\Gamma$ be its closest point. 
The change of variables formula for surface measures reads
\begin{align}
\mu_h(x)\,ds_h(x) = ds(\mathbf{p}(x)), \label{eq_geo_eqv}
\end{align}
where $\mu_h$ is the Jacobian factor of the closest-point map and $d s_h(x)$ and $d s(p(x))$ are the surface measures on $\Gamma_h$ and $\Gamma$ respectively.
This identity implies the geometry approximation
\begin{equation}
\|1-\mu_h\|_{L^\infty(\Gamma_h)} \le C\,h^{q+1},
\qquad
\|v^e\|_{\Gamma_h} \le C\,\|v\|_{\Gamma}, 
\label{eq_geo_eqv_1}
\end{equation}
for all sufficiently small~$h$.
These estimates are used repeatedly in the error analysis, see also~\cite{reusken2015analysis}.

Finally, for any smooth scalar function $f_0'$ and $v:\Gamma\to\mathbb{R}$, 
the normal extension commutes with composition:
\begin{equation}
\big(f_0'\!\circ v\big)^{e} = f_0'\!\circ (v^{e}).
\label{eq_comp_ext}
\end{equation}
\end{itemize}

Let us introduce the following bilinear forms on $H^1(\Gamma_h)\cap H^1(\Omega^{\Gamma}_h)\times H^1(\Gamma_h)\cap H^1(\Omega^{\Gamma}_h)$:
\begin{align}\label{bilinear}
	a_h(u, v) &:=  (\nabla_{\Gamma_h}u, \nabla_{\Gamma_h} v)_{\Gamma_h} + h \int_{\Omega^{\Gamma}_h} (\mathbf{n}_h \cdot \nabla u) (\mathbf{n}_h \cdot \nabla v) \,dx. 
\end{align}
Here, the volumetric term in  bilinear form is stabilization term that ensure well-posedness of the discrete problem and control the extension of solutions from the surface to the bulk domain~\cite{GLR:2018,LOX:2018}.

We define the stabilized Ritz projection $\pi_h: H^1(\Gamma)\to V_h$ as the solution to
\begin{align} 
a_h(\pi_h v, w_h) &= (\nabla_\Gamma v, \nabla_\Gamma w_h^\ell)\quad\forall\, w_h\in V_h, 
\qquad (\pi_h v - v^e, 1)_{\Gamma_h}=0.
\label{proj_opera}
\end{align}
From the error estimates for the trace finite element solution to the Laplace--Beltrami problem in~\cite{GLR:2018, reusken2015analysis}, it follows 
that  
\begin{align}
\|\pi_h v\|_{H^1(\Gamma_h)} &\le C \|v\|_{H^1(\Gamma)}, \nonumber \\
\|\pi_h v - v^e\|_{\Gamma_h} + h \|\nabla_{\Gamma_h}(\pi_h v - v^e)\|_{\Gamma_h}
&\le C (h^{r+1}+h^{q+1})\, \|v\|_{H^{r+1}(\Gamma)}, 
\label{proj_est}
\end{align}
for $v\in H^{r+1}(\Gamma), ~ r=1,\dots,m,$ and with  constants $C$ independent of how $\Gamma$ or $\Gamma_h$ cuts through the background mesh.

The geometric error term $h^{q+1}$ reflects the accuracy of the surface approximation, while the approximation error term $h^{k+1}$ results from the finite element interpolation properties. For optimal convergence, one should choose $q = m$ to balance these error contributions.
\smallskip

For time discretization, we let $\tau$ be the time step size and denote $\delta v^{n+1}=v^{n+1}-v^{n}$. At time instance $t_n = n\tau$, with time step $\tau = T/N$, $c^n$ denotes the approximation of $c(t^n, x)$; similar notation is used for other quantities of interest.
  
With these preliminary results collected, we now proceed to derive stability and error estimates for the numerical solution of the surface Cahn--Hilliard problem.

\section{Finite element method and stability estimate} \label{sec:stability}
In this section, we introduce a finite element method and establish its  stability property. 

The variational problem (\ref{eq_1})--(\ref{eq_3}) is discretized in space by TraceFEM and in time by the semi-implicit Euler scheme. The discrete problem reads: Given $c^n_h\in V_h$ that approximates $c(t_n)$  find $(c_h^{n+1}, \mu^{n+1}_h) \in V_h \times V_h$ such that
\begin{align} \label{dis_ch}
	\Bigl(\frac{c_h^{n+1}-c_h^{n}}{\tau}, v_h\Bigr)_{\Gamma_h} &= -Ma_h(\mu^{n+1}_h, v_h), \nonumber\\
	(\mu^{n+1}_h, q_h)_{\Gamma_h} &= \beta_s \tau \Bigl(\frac{c_h^{n+1}-c_h^{n}}{\tau}, q_h\Bigr)_{\Gamma_h} + ({f'_0}(c^{n}_h), q_h)_{\Gamma_h} + \epsilon^2a_h(c^{n+1}_h, q_h),
\end{align}
for all $(v_h, q_h) \in V_h \times V_h$, with initial condition $c_h^{0} = \pi_h c({t^0})$. The bilinear form $a_h(\cdot, \cdot)$ is defined in (\ref{bilinear}).
The term containing $\beta_s \tau (c_h^{n+1}-c_h^{n})$ is a stabilizing term that relaxes an otherwise severe time-step constraint from the explicit treatment of the nonlinearity while maintaining first-order temporal accuracy~\cite{shen2010numerical}.   

We now show the stability of the numerical scheme. 
For the purpose of analysis and following \cite{shen2010numerical}, for a given fixed $K>1$, we work with a regularized version of the potential 
\begin{align}
	f_0(c) = \frac{1}{4}(c^2-1)^2, \label{eq_5}
\end{align}  defined as:
\begin{equation}\label{fStab}
{f_0}(c) = 
\begin{cases}
	\frac{1}{4}(c^2 - 1)^2, & c \in [-K, K] \\
	\frac{1}{2}(3K^2 - 1)c^2 - 2K^3 c + \frac{1}{4}(3K^4 + 1), & c > K \\
	\frac{1}{2}(3K^2 - 1)c^2 + 2K^3 c + \frac{1}{4}(3K^4 + 1), & c < -K
\end{cases}
\end{equation}
Its derivative ${f'}_0(c)$ satisfies the growth conditions with $L = 3K^2 - 1$:
\begin{align}\label{eq_1124_1}
|f'_0(c)| \leq L|c|, \quad
f'_0(c)c \geq -c, 
\end{align}
and Lipschitz condition:
\begin{align}\label{eq_1124}
-1 \leq \frac{f'_0(x) - f'_0(y)}{x - y} \leq L, \quad \forall x, y \in \mathbb{R}, \; x \neq y.
\end{align}
This regularized potential is $C^2$ with globally bounded second derivatives globally: There exists a constant $L > 0$ such that 
\begin{align}\label{bound_f}
	\max_{x\in \mathbb{R}}| f_0''(x) |\le L.
\end{align}

\begin{theorem}
	Let 
	\begin{align*}  
		E^{n+1}_h=  \frac{\epsilon^2}{2} a_h( c^{n+1}_h,  c^{n+1}_h)   +\int_{\Gamma_h} f_0(c^{n+1}_h)\,dx,
	\end{align*}
	be the modified discrete energy. Under the condition 
	\begin{align}\label{eq_21}
		\beta_s \ge \frac{L}{2},  
	\end{align}
	the system (\ref{dis_ch}) obeys the following energy dissipation law
	\begin{align} \label{eq17}
		E^{n+1}_h-E^{n}_h +\frac{\epsilon^2}{2} a_h(\delta c^{n+1}_h,\delta c^{n+1}_h)  + M\tau a_h(\mu_h^{n+1},\mu_h^{n+1})\le 0.
	\end{align}
	In particular, this implies that the scheme (\ref{dis_ch}) is energy stable in the sense that $E^{n+1}_h \le E^{n}_h$  (the discrete analogue of (\ref{eq_7})) for all $n = 0, 1, 2,\dots $.
\end{theorem}

\begin{proof}
	We test the system (\ref{dis_ch}) with $v_h=\mu^{n+1}_h$ and $q_h= \delta c_h^{n+1} $ to obtain
	\begin{align}\label{eq_10}
		& \underbrace{\beta_s \norm{\delta c_h^{n+1}}^2_{\Gamma_h}}_{T_1} +\underbrace{({f'_0}(c^{n}_h), \delta c_h^{n+1})_{\Gamma_h}}_{T_2} +\underbrace{\epsilon^2 a_h(c^{n+1}_h, \delta c_h^{n+1})}_{T_3}    + M\tau a_h(\mu_h^{n+1},\mu_h^{n+1}) =0. 
	\end{align}
	Let us handle the terms on the LHS of (\ref{eq_10}) one by one. 
	Thanks to the Taylor's expansion
	\begin{align*}
		f_0(c^{n+1}_h)-f_0(c^n_h)=f_0'(c^n_h) \delta  c^{n+1}_h+\frac{f''_0(\xi)}{2}(\delta c^{n+1}_h)^2,\quad \xi\in (c^{n}_h, c^{n+1}_h),
	\end{align*}
	 the $T_2$ term can be written as:
	\begin{align*} 
		T_2= (f_0(c^{n+1}_h)-f_0(c^n_h),1)_{\Gamma_h} - \frac{1}{2}( f''_0(\xi) \delta c^{n+1}_h,\delta c^{n+1}_h)_{\Gamma_h}.
	\end{align*}
	Using  a standard polarization inequality in the third   term   of (\ref{eq_10}), we have: 
	\begin{align}
		T_3 =&\frac{\epsilon^2}{2}\Big(a_h(c^{n+1}_h, c^{n+1}_h)-a_h(c^{n}_h, c^{n}_h)+a_h(\delta {c}^{n+1}_h, \delta {c}^{n+1}_h)\Big). 
	\end{align}
	Therefore, we get
	\begin{align}\label{eq_uuu}
		&  T_1+T_2+T_3\nonumber \\ & =\beta_s \norm{\delta c_h^{n+1}}^2_{\Gamma_h}+(f_0(c^{n+1}_h)-f_0(c^n_h),1)_{\Gamma_h} - \frac{1}{2}( f''_0(\xi_1) \delta c^{n+1}_h,\delta c^{n+1}_h)_{\Gamma_h} \nonumber \\&+\frac{\epsilon^2}{2}\Big(a_h(c^{n+1}_h, c^{n+1}_h)-a_h(c^{n}_h, c^{n}_h)+a_h(\delta {c}^{n+1}_h, \delta {c}^{n+1}_h)\Big).
	\end{align}
	Recalling the definition of $E^{n}_h$ in (\ref{eq_uuu}), the above expression can be rewritten as:
	\begin{align*}
		&T_1+T_2+T_3 \\ = & E^{n+1}_h-E^{n}_h + \beta_s \norm{\delta c^{n+1}_h}^2_{\Gamma_h}+\frac{\epsilon^2}{2}a_h(\delta c^{n+1}_h , \delta c^{n+1}_h )  - \frac{1}{2}( f''_0(\xi_1) \delta c^{n+1}_h,\delta c^{n+1}_h)_{\Gamma_h}.
	\end{align*}
	Therefore, (\ref{eq_10}) becomes
	\begin{multline}\label{eq_20}
		   E^{n+1}_h-E^{n}_h + \beta_s \norm{\delta c^{n+1}_h}^2_{\Gamma_h}  +\frac{\epsilon^2}{2}a_h(\delta c^{n+1}_h , \delta c^{n+1}_h ) +\tau M a_h(\mu_h^{n+1},\mu_h^{n+1}) \\=\frac{1}{2}( f''_0(\xi_1) \delta c^{n+1}_h,\delta c^{n+1}_h)_{\Gamma_h}.  
	\end{multline}
	Using  (\ref{bound_f}) in the right-hand side of (\ref{eq_20}) we have:
	\begin{align}\label{eq_23}
		\frac{1}{2}( f''_0(\xi_1) \delta c^{n+1}_h,\delta c^{n+1}_h)_{\Gamma_h} \le \frac{L}{2} \norm{\delta c^{n+1}_h}^2_{\Gamma_h}.
	\end{align}
	Thanks to the condition (\ref{eq_21}), eqs. \eqref{eq_20} and \eqref{eq_23} imply the dissipative law (\ref{eq17}).  \\ 
	
\end{proof}

\section{Error Analysis} \label{sec:error_analysis}

This section establishes a priori error estimates for the FE method \eqref{dis_ch}. 
The analysis accounts for errors arising from  the spatial discretization, geometry approximation,  and the temporal discretization. It will be demonstrated that under appropriate regularity assumptions on the exact solution, the proposed method achieves an optimal convergence order.

Let us  introduce some further  notation:
\begin{alignat*}{2}
	e^{n}_h &= \pi_h c(t^{n}) - c_h^{n}, \quad & \eta^{n} &= c^e(t^{n}) - \pi_h c(t^{n}),\quad n=0,1,\dots, \\
	\zeta^{n}_h &= \pi_h \mu(t^{n}) - \mu_h^{n}, \quad & \kappa^{n} &= \mu^e(t^{n}) - \pi_h \mu(t^{n})),\quad n=1,2,\dots.
\end{alignat*}
so that the FE error in the concentration and chemical potential is split into the FE and interpolation errors. Approximation properties of the stanilized Ritz projection are sufficient for taking care of estimating various norms of $\eta$-s and $\kappa$-s. We need to focus on bounding norms of $e^{n}_h$ and $\zeta^{n}_h$.  

Using the  weak formulation (\ref{eq_8}) and the discrete scheme (\ref{dis_ch}), we get the error equations:
\begin{align} \label{dis_error}
	\frac{1}{\tau} (\delta e_h^{n+1}, v_h)_{\Gamma_h} + M a_h(\zeta^{n+1}_h, v_h) &= L^{n+1}_h(v_h), \nonumber \\
	-(\zeta^{n+1}_h, q_h)_{\Gamma_h}+ \beta_s (\delta e_h^{n+1}, q_h)_{\Gamma_h} + \epsilon^2 a_h(e^{n+1}_h, q_h) &\\
    +({f'_0}(\pi_h c^{n})-{f'_0}(c_h^{n}), q_h)_{\Gamma_h}&= \tilde{L}^{n+1}_h(q_h),\quad \forall\, v_h,q_h\in V_h.
\end{align}
By the definition of $c_h^0$, we have $e_h^0 = 0$. 
The residual terms are
\begin{equation}\label{l_eq}
	L_h^{n+1}(v_h) := \sum^{3}_{j=1} \Psi_j^{n+1}(v_h), \quad \tilde{L}_h^{n+1}(q_h) := \sum^{4}_{j=1} \Phi_j^{n+1}(q_h),
\end{equation}
where the individual components are given by:
\begin{align*}
	\Psi_1^{n+1}(v_h) &:= \left(\frac{(c^{n+1}-c^{n})^e}{\tau}, v_h\right)_{\Gamma_h} - \left( \frac{\partial c^{n+1}}{\partial t}, v_h^\ell \right)_{\Gamma}, \\
	\Psi_2^{n+1}(v_h) &:= \frac{1}{\tau}\left((\pi_h-I)\delta (c^{n+1})^e, v_h\right)_{\Gamma_h}, \\
	\Psi_3^{n+1}(v_h) &:= M a_h(\pi_h \mu^{n+1},v_h) - M (\nabla_{\Gamma} \mu^{n+1}, \nabla_{\Gamma} v_h^\ell)_{\Gamma_h}=0, \\
	\Phi_1^{n+1}(q_h) &:= -(\pi_h \mu^{n+1}, q_h)_{\Gamma_h} + (\mu^{n+1}, q_h^\ell)_{\Gamma}, \\
	\Phi_2^{n+1}(q_h) &:= \beta_s (\pi_h (c^{n+1}-c^n), q_h)_{\Gamma_h}, \\
	\Phi_{3,a}^{n+1}(q_h) &:= (({f'_0}(c^{n}))^e, q_h)_{\Gamma_h} - ({f'_0}(c^{n+1}), q_h^\ell)_{\Gamma}, \\
    \Phi_{3,b}^{n+1}(q_h) &:= -(({f'_0}(c^{n}))^e-{f'_0}(\pi_h c^{n}), q_h)_{\Gamma_h}, \\
	\Phi_4^{n+1}(q_h) &:= \epsilon^2 a_h( \pi_h c^{n+1},q_h) -\epsilon^2 (\nabla_{\Gamma} c^{n+1}, \nabla_{\Gamma} q_h^\ell)_{\Gamma}=0. 
\end{align*}
Note that $\Psi_3$ and $\Phi_4$ terms above vanish due to the definition of $\pi_h$ projection (\ref{proj_opera}).

The main result of this section is the following convergence theorem, which provides optimal-order error estimates under appropriate regularity assumptions on the exact solution. 

\begin{theorem}\label{error_analysis_1}
	Let $T = N\tau > 0$ and assume $(c, \mu)$ is the solution of (\ref{eq_8}) with sufficient regularity: 
	\begin{align*}
			c &\in L^2(0,T;H^{m+1}(\Gamma)),~~ 
			c_t \in L^2(0,T;H^{m+1}(\Gamma)),~~c_{tt}\in L^{\infty}(0,T;L^2(\Gamma))  \\
			\mu &\in L^2(0,T;H^{m+1}(\Gamma)),
	\end{align*}
	and $\Gamma$ sufficiently smooth. 
    Let $(c_h^{n}, \mu_h^n)$, $n = 1, 2, \ldots, N-1$, be the finite element solution of (\ref{dis_ch}). Then the following error estimate holds:
	\begin{align}\label{eq_06_1}
		\norm{ c^e(t^{N}) - c_h^{N}}^2_{\Gamma_h} 
		&+ \frac{M\tau}{\epsilon^2} \sum_{n=1}^{N-1} \norm{ \mu^e(t^{n+1}) - \mu_h^{n+1}}^2_{\Gamma_h} \nonumber \\
		&\le C (1 + T\alpha_1 \exp(T\alpha_1)) (K_1 \tau^2 + K_2 h^{2\min\{m,q\}+1}),
	\end{align}
	where $\alpha_1$ is defined in (\ref{grow_1}), and the constants $K_1$ and $K_2$ are given by:
	\begin{align}\label{gron_coff}
		K_1 &= C\tau {\norm{c_{tt}}_{L^{\infty}(0,T;L^2(\Gamma))}^2} + \frac{M(\beta^2_s +\tau L^2)}{\epsilon^2} \norm{c_t}^2_{L^2(0,T; L^{2}(\Gamma))}, \\
		K_2 &= \left(C + \frac{M\tau^2\beta^2_s }{\epsilon^2}\right) \norm{c_t}^2_{L^2(0,T; H^{m+1}(\Gamma))} 
		+ C\tau {\norm{c_t}_{L^{2}(0,T;L^2(\Gamma))}^2} \nonumber \\
		&\quad + \frac{CML^2\tau}{\epsilon^2} \norm{c}^2_{L^2(0,T; L^{2}(\Gamma))} + C\frac{M\tau}{\epsilon^2} {\norm{\mu}_{L^2(0,T;H^{m+1}(\Gamma))}^2 }
		 \nonumber \\
		&\quad  + C\frac{M\tau}{\epsilon^2}\norm{\mu}^2_{L
			^2((0,T);L^{2}(\Gamma))}+\frac{C ML^2\tau }{\epsilon^2}   \|c\|^2_{(L^2({0,T});H^{m+1}(\Gamma))}.
	\end{align}
\end{theorem}
\begin{proof}
	Taking $v_h= 2 \tau e^{n+1}_h$ and $q_h= \frac{-2 M \tau}{\epsilon^2} \zeta^{n+1}_h$ in (\ref{dis_error}) and summing up equalities we obtained that
	\begin{align}\label{error_2_1}
		\norm{ e^{n+1}_h}_{\Gamma_h}^2-\norm{ e^{n}_h}_{\Gamma_h}^2+ \norm{\delta e^{n+1}_h}_{\Gamma_h}^2 &+ \frac{2 M\tau }{\epsilon^2} \norm{ \zeta^{n+1}_h}^2_{\Gamma_h}   
		=\sum^{7}_{j=1} R_j^{n+1}
	\end{align}
	where the residual terms are defined as:
	\begin{align*}
		R_1^{n+1} &:= 2\tau \Psi_1(e^{n+1}_h),~ 
		R_2^{n+1} := 2\tau\Psi_2(e^{n+1}_h), ~
		R_3^{n+1} := -\frac{2M\tau}{\epsilon^2}\Phi_1(\zeta^{n+1}_h), \\
		R_4^{n+1} &:= -\frac{2M\tau}{\epsilon^2}\Phi_2(\zeta^{n+1}_h),~ 
		R_5^{n+1} := -\frac{2M\tau}{\epsilon^2}\Phi_{3,a}(\zeta^{n+1}_h),~ 
        R_6^{n+1} := -\frac{2M\tau}{\epsilon^2}\Phi_{3,b}(\zeta^{n+1}_h), \\
		R_7^{n+1} &:= \frac{2M\beta_s\tau}{\epsilon^2} (\delta e_h^{n+1}, \zeta^{n+1}_h)_{\Gamma_h}.
	\end{align*}
	We now estimate the first residual term $R_1^{n+1}$:
	\begin{align*}
		R_1^{n+1} = \int_{\Gamma_h} \frac{(c^{n+1}-c^{n})^e}{\tau} v_h \,ds- \int_{\Gamma}  \frac{\partial c^{n+1}}{\partial t}  v^\ell_h \,ds.
	\end{align*}
	Applying Taylor's theorem, we have: 
	\begin{align}\label{taylor}
		c(t) &= c(t^{n+1}) +c_t(t^{n+1})(t-t^{n+1})+\int^t_{t^{n+1}}c_{tt}(s)(t-s)\,ds. 
	\end{align}
	Using expansion (\ref{taylor}), the term $R_1^{n+1}$ can be rewritten as:
	\begin{align}\label{eq_28_1}
		R_1^{n+1}  &=-\int^{t^{n+1}}_{t^n}\int_{\Gamma_h}c^e_{tt}(t)\frac{(t-t^{n})}{\tau}\, v_h \,ds \,dt+\int_{\Gamma_h}  (c_t^{n+1})^e\, v_h \,ds-\int_{\Gamma} c_t^{n+1}  v^\ell_h \,ds. 
	\end{align}
	We estimate each term in (\ref{eq_28_1}) separately. Using the Cauchy--Schwarz inequality and (\ref{eq_geo_eqv_1}) in the first term, we have:
	\begin{align*}
		\Bigl| -\int^{t^{n+1}}_{t^n} \int_{\Gamma_h}c^e_{tt}(t)\frac{(t-t^{n})}{\tau}\, v_h \,ds\,dt \Bigr| & \le \frac{1}{2} \tau \sup_{t\in[0,T]} \norm{c^e_{tt}(t)}_{\Gamma_h} \norm{v_h}_{\Gamma_h} \\
		&\le C \tau \norm{c_{tt}}_{L^\infty(0,T;L^2(\Gamma))} \norm{v_h}_{\Gamma_h},
	\end{align*}
	and using the geometric consistency for stationary surfaces \cite[Section 5]{reusken2015analysis}:

\begin{align}\label{geo_con}
\int_{\Gamma_h} (c_t^{n+1})^e\, v_h \, ds_h \;-\; \int_{\Gamma} c_t^{n+1}\, v_h^\ell \, ds
&= \int_{\Gamma_h} (c_t^{n+1})^e\,(\mu_h - 1)\, v_h \, ds_h \nonumber\\
&\le \|1-\mu_h\|_{L^{\infty}( \Gamma_h)}\,\|(c_t^{n+1})^e\|_{\Gamma_h}\,\|v_h\|_{\Gamma_h} \nonumber\\
&\le C\,h^{q+1}\,\|c_t^{n+1}\|_{\Gamma}\,\|v_h\|_{\Gamma_h}.
\end{align}
In the last estimate we used the geometry approximation and the norm equivalence (\ref{eq_geo_eqv_1}),
which follow from the closest-point map identity  (\ref{eq_geo_eqv}), see \cite{reusken2015analysis}.

Putting it together and choosing $v_h=\tau e^{n+1}_h$ we have:
	\begin{align*}
		R_1^{n+1} &\le C \tau(\tau^2\norm{c_{tt}}_{L^{\infty}(0,T;L^2(\Gamma))}^2+h^{2q+2}\|c_t^{n+1}\|^2_{\Gamma})  + \frac{\tau}{4}\norm{e^{n+1}_h}^2_{\Gamma_h}.
	\end{align*}

Consider the second term $R_2^{n+1}$.
Let $\delta c^{n+1}:=c(t^{n+1})-c(t^{n})=\int_{t^n}^{t^{n+1}} c_t(t)\,dt$ and
\[
R_2^{n+1}:=2\big((\pi_h-I)\,\delta c^{n+1},\, e_h^{n+1}\big)_{\Gamma_h}.
\]
Then by Cauchy–Schwarz in time and space, and the $L^2$-error estimate for the stabilized Ritz projector  (\ref{proj_est}),
we get
\begin{align*}
|R_2^{n+1}|
&\le 2\,\|(\pi_h-I)\delta c^{n+1}\|_{\Gamma_h}\,\|e_h^{n+1}\|_{\Gamma_h}\\
&\le 2\sqrt{\tau}\Big(\int_{t^n}^{t^{n+1}}\!\!\|(\pi_h-I)c_t(t)\|_{\Gamma_h}^2\,dt\Big)^{1/2}
          \,\|e_h^{n+1}\|_{\Gamma_h}\\
&\le C\int_{t^n}^{t^{n+1}}\!\!\|(\pi_h-I)c_t(t)\|_{\Gamma_h}^2\,dt
   + \frac{\tau}{4}\,\|e_h^{n+1}\|_{\Gamma_h}^2\\
&\le C(h^{m+1}+h^{q+1})\int_{t^n}^{t^{n+1}}\!\!\|c_t(t)\|_{H^{m+1}(\Gamma)}^2\,dt
   + \frac{\tau}{4}\,\|e_h^{n+1}\|_{\Gamma_h}^2 .
\end{align*}
Here $(\pi_h-I)c_t(t)=\pi_h(c_t(t))-c^e_t(t)$,
and the use of the extension ensures that all quantities are defined on~$\Gamma_h$.

  Next,   we proceed by decomposing $\Phi_1^{n+1}(q_h)$ as follows:
    \begin{align*}
        \Phi_1^{n+1}(q_h) &:= -(\pi_h \mu^{n+1}, q_h)_{\Gamma_h} + (\mu^{n+1}, q_h^\ell)_{\Gamma}\\
          &= ((\mu^{n+1})^e-\pi_h \mu^{n+1}, q_h)_{\Gamma_h}-((\mu^{n+1})^e, q_h)_{\Gamma_h} + (\mu^{n+1}, q_h^\ell)_{\Gamma} \\
          &= ((\mu^{n+1})^e-\pi_h \mu^{n+1}, q_h)_{\Gamma_h}-  \int_{\Gamma_h} (\mu^{n+1})^e(1 - \mu_h) q_h \, ds_h.
    \end{align*}
  For the projection error term, we employ the $L^2$-error estimate again:
\begin{align*}
((\mu^{n+1})^e-\pi_h \mu^{n+1}, q_h)_{\Gamma_h}&\le C\|(\mu^{n+1})^e-\pi_h \mu^{n+1}\|_{\Gamma_h}\|q_h\|_{\Gamma_h}\\ 
&\le C( h^{m+1}+ h^{q+1}) \|\mu^{n+1}\|_{H^{m+1}(\Gamma)}\|q_h\|_{\Gamma_h}.
\end{align*}
For the geometric consistency term, following similar arguments as in \eqref{geo_con}, we obtain:
\begin{align*}
    \int_{\Gamma_h} (\mu^{n+1})^e(1 - \mu_h) q_h \, ds_h &\le C h^{q+1} \norm{\mu^{n+1}}_{\Gamma} \norm{q_h}_{\Gamma_h}.
\end{align*}
Combining these estimates yields the desired bound:
\begin{align*}
   |\Phi_1^{n+1}(q_h) | \le C((h^{m+1}+ h^{q+1})\|\mu^{n+1}\|_{H^{m+1}(\Gamma)} +h^{q+1} \|\mu^{n+1}\|_{\Gamma}) \|q_h\|_{\Gamma_h}.
\end{align*}
Therefore, the remainder term $R_3^{n+1}$ satisfies:
\begin{align*}
    R_3^{n+1} \le \frac{ CM\tau }{\epsilon^2}(h^{2m+2}+ h^{2q+2})\|\mu^{n+1}\|^2_{H^{m+1}(\Gamma)} +\frac{ M\tau }{8\epsilon^2} \norm{\zeta_h^{n+1}}^2_{\Gamma_h}.
\end{align*}

Let us now turn to estimating $R_4^{n+1}$.
\[
R_4^{n+1}:=-\frac{2M\tau}{\epsilon^2}\,\beta_s\,(\pi_h(c^{n+1}-c^{n}),\zeta_h^{n+1})_{\Gamma_h}.
\]
Using $c^{n+1}-c^{n}=\int_{t^n}^{t^{n+1}} c_t(t)\,dt$, we write
\[
R_4^{n+1}
=-\frac{2M\tau}{\epsilon^2}\,\beta_s\!\int_{t^n}^{t^{n+1}}\!(\pi_h c_t(t),\zeta_h^{n+1})_{\Gamma_h}\,dt.
\]
By the Cauchy-Schwarz inequality in space and time, followed by Young's inequality,
\begin{align*}
|R_4^{n+1}|
&\le \frac{2M\tau}{\epsilon^2}\,\beta_s
     \int_{t^n}^{t^{n+1}}\!\!\|\pi_h c_t(t)\|_{\Gamma_h}\,\|\zeta_h^{n+1}\|_{\Gamma_h}\,dt\\
&\le \frac{2M\tau}{\epsilon^2}\,\beta_s\,
     \sqrt{\tau}\Big(\int_{t^n}^{t^{n+1}}\!\!\|\pi_h c_t(t)\|_{\Gamma_h}^{2}\,dt\Big)^{1/2}
     \,\|\zeta_h^{n+1}\|_{\Gamma_h}\\
     &\le \frac{M\tau}{8\epsilon^2}\,\|\zeta_h^{n+1}\|_{\Gamma_h}^{2}
\;+\; \frac{CM\,\tau^2\,\beta_s^{\,2}}{\epsilon^2}
      \int_{t^n}^{t^{n+1}}\!\!\|\pi_h c_t(t)\|_{\Gamma_h}^{2}\,dt.
\end{align*}
Finally, split $\pi_h c_t = (\pi_h-I)c_t + c_t^e$ and use $(a+b)^2\le 2a^2+2b^2$ to obtain:
\begin{align*}
|R_4^{n+1}|
&\le \frac{M\tau}{8\epsilon^2}\,\|\zeta_h^{n+1}\|_{\Gamma_h}^{2}
 + \frac{C M\,\tau^2\,\beta_s^{\,2}}{\epsilon^2}
   \int_{t^n}^{t^{n+1}}\!\!\Big(\|(\pi_h-I)c_t(t)\|_{\Gamma_h}^{2}
                               + \|c_t^e(t)\|_{\Gamma_h}^{2}\Big)\,dt.
\end{align*}
Using the projection estimates (\ref{proj_est}) and the norm equivalence (\ref{eq_geo_eqv_1}), we obtain:
\begin{align*}
|R_4^{n+1}|
&\le \frac{M\tau}{8\epsilon^2}\,\|\zeta_h^{n+1}\|_{\Gamma_h}^{2}
 + \frac{C\,M\,\tau^2\,\beta_s^{\,2}}{\epsilon^2}\,
   \Big( h^{2m+2}+h^{2q+2} \Big)
   \int_{t^n}^{t^{n+1}}\!\!\|c_t(t)\|_{H^{m+1}(\Gamma)}^{2}\,dt
\\&\qquad
 + \frac{C\,M\,\tau^2\,\beta_s^{\,2}}{\epsilon^2}
   \int_{t^n}^{t^{n+1}}\!\!\|c_t(t)\|_{\Gamma}^{2}\,dt .
\end{align*}
	 Using (\ref{eq_comp_ext}), (\ref{eq_1124_1}) and (\ref{eq_1124}) and the similar arguments as in $\Phi_1^{n+1}(q_h)$, we have:
    \begin{align}\label{eq_1013}
   \Phi_{3,a}^{n+1}(q_h)= &\left| \int_{\Gamma_h} (f'_0(c^{n}))^e  q_h \, ds_h 
- \int_{\Gamma} f'_0(c^{n+1}) q^\ell_{h} \, ds \right| \nonumber\\
=&\left| \int_{\Gamma_h} ((f'_0(c^{n}))^e - (f'_0(c^{n+1}))^e) q_h \, ds_h 
+ \int_{\Gamma_h} (f'_0(c^{n+1}))^e q_h \, ds_h 
- \int_{\Gamma} f'_0(c^{n+1}) q^\ell_{h} \, ds \right| \nonumber\\
\leq& L \int_{\Gamma_h} |(c^{n} - c^{n+1})^e q_h| \, ds_h 
+ \left| \int_{\Gamma_h} (f'_0(c^{n+1}))^e(1 - \mu_h) q_h \, ds_h \right| \nonumber\\
\leq& L \tau \|c_t\|_{L^\infty((0,T),L^2(\Gamma))} \|q_h\|_{\Gamma_h} 
+ L \|c^{n+1}\|_{\Gamma} h^{q+1} \|q_h\|_{\Gamma_h}.
\end{align}
Here $L$ is the Lipschitz constant of $f_0'$.

Choosing $q_h=-\frac{2M\tau}{\epsilon^2}\,\zeta_h^{n+1}$ in \eqref{eq_1013} and applying Young’s inequality yields:
	\begin{align}\label{eq_57}
		{ |R_5^{n+1}|}\leq& \frac{C ML^2}{\epsilon^2}  \tau^{3} \|c_t\|^2_{L^\infty((0,T),L^2(\Gamma))} +\frac{C ML^2\tau h^{2q+2}}{\epsilon^2}   \|c^{n+1}\|^2_{\Gamma}
 +\frac{ M\tau }{8\epsilon^2} \norm{\zeta_h^{n+1}}^2_{\Gamma_h}. 
	\end{align}

The next  term is estimated using the Cauchy--Schwarz inequality, Lipschitz continuity of ${f'_0}$,  and (\ref{proj_est}):
\begin{align*}
\big|\big((f_0'(c^{n}))^e-f_0'(\pi_h c^{\,n}),\,q_h\big)_{\Gamma_h}\big|
&\le \|(f_0'(c^{n}))^e-f_0'(\pi_h c^{\,n})\|_{\Gamma_h}\,\|q_h\|_{\Gamma_h}\\
&\le L\,\|(c^{n})^e-\pi_h c^{\,n}\|_{\Gamma_h}\,\|q_h\|_{\Gamma_h} \\
&\le C\,L\,(h^{m+1}+h^{q+1})\,\|c^{n}\|_{H^{m+1}(\Gamma)}\,\|q_h\|_{\Gamma_h}\; .
\end{align*}

Choosing $q_h=-\frac{2M\tau}{\epsilon^2}\,\zeta_h^{n+1}$ in \eqref{eq_1013} and applying Young’s inequality yields:
\begin{align}\label{eq_57_1}
		{ |R_6^{n+1}|}\leq& \frac{C ML^2\tau (h^{2m+2}+h^{2q+2})}{\epsilon^2}   \|c^n\|^2_{H^{m+1}(\Gamma)}
 +\frac{ M\tau }{8\epsilon^2} \norm{\zeta_h^{n+1}}^2_{\Gamma_h}. 
\end{align}
	The $R_{6}^{n+1}$ term is estimated by  using the Cauchy--Schwarz inequality and the Young's inequality:
	\begin{align*}
		R_{7}^{n+1}&\le \frac{16 M \beta^2_s \tau}{\epsilon^2} \norm{\delta e_h^{n+1}}^2_{\Gamma_h} + \frac{ M\tau}{8\epsilon^2}\norm{\zeta^{n+1}_h}^2_{\Gamma_h}\\
		&\le \frac{ 16 M \beta^2_s\tau}{\epsilon^2} \norm{ e_h^{n+1}}^2_{\Gamma_h} + \frac{16 M \beta^2_s \tau}{\epsilon^2} \norm{ e_h^{n}}^2_{\Gamma_h} + \frac{ M  \tau}{8\epsilon^2}\norm{\zeta^{n+1}_h}^2_{\Gamma_h}.
	\end{align*}

  	Summing over $n = 0, 1, 2, \ldots, N-1$ in equation (\ref{error_2_1}) and using the fact that $e^{0}_h = 0$, we obtain 
 	\begin{align*}
		\norm{ e^N_h}_{\Gamma_h}^2& + \sum_{n=1}^{N-1}\norm{\delta e^{n+1}_h}_{\Gamma_h}^2 + \frac{2 M\tau }{\epsilon^2} \sum_{n=1}^{N-1}\norm{ \zeta^{n+1}_h}^2_{\Gamma_h}   
		=\sum_{n=1}^{N-1}\sum^{7}_{j=1} R_j^{n+1} \\
		\le & C \tau(\tau^2\norm{c_{tt}}_{L^{\infty}(0,T;L^2(\Gamma))}^2+h^{2q+2}\norm{c_t}^2_{L^2(({0,T});L^2(\Gamma))}\\
        & + C (h^{2m+2} +h^{2q+2})\norm{c_t}^2_{L^2(({0,T});H^{m+1}(\Gamma))}   
        \\&+C\frac{ M\tau }{\epsilon^2} h^{2q+2} \norm{\mu}^2_{L^2((0,T);L^{2}(\Gamma))}+C\frac{ M\tau }{\epsilon^2}(h^{2m+2}+h^{2q+2}) {\norm{\mu}_{L^2(0,T;H^{m+1}(\Gamma))}^2 }   
        \\&+\frac{ M \beta^2_s }{\epsilon^2}\tau^2(h^{2m+2}+h^{2q+2})  \norm{c_t}^2_{L^2(({0,T});H^{m+1}(\Gamma))} +\frac{ M \beta^2_s }{\epsilon^2}\tau^2 \norm{c_t}^2_{L^2(({0,T});L^{2}(\Gamma))}   
        \\&+\frac{ CM L^2}{\epsilon^2}\tau^3   \norm{c_t}^2_{L^2(({0,T});L^2(\Gamma))} +CML^2\tau h^{2q+2} \norm{c}^2_{L^2(({0,T});L^2(\Gamma))} 
        \\&+\frac{C ML^2\tau (h^{2m+2}+h^{2q+2})}{\epsilon^2}   \|c\|^2_{L^2(({0,T});H^{m+1}(\Gamma))}+ \alpha_1 \tau \sum_{n=1}^{N-1}\norm{ e^{n+1}_h}_{\Gamma_h}^2\\
        &+ \alpha_2\frac{2 M\tau }{\epsilon^2} \sum_{n=1}^{N-1}\norm{ \zeta^{n+1}_h}_{\Gamma_h}, 
    \end{align*}
	where the constants are defined as:
	\begin{align}\label{grow_1}
		\alpha_1 &:= \frac{1}{2}+\frac{32M\beta^2_s}{\epsilon^2},\quad
		\alpha_2 :=\frac{5}{16}.
	\end{align}
		 The right-hand side term containing $\zeta$ is absorbed into the left-hand side term. Therefore, the inequality simplifies to:
		\begin{align}\label{error_2_2}
			&\norm{ e^N_h}_{\Gamma_h}^2+ \sum_{n=1}^{N-1}\norm{\delta e^{n+1}_h}_{\Gamma_h}^2 + \frac{M\tau }{\epsilon^2} \sum_{n=1}^{N-1}\norm{ \zeta^{n+1}_h}^2_{\Gamma_h}   \nonumber \\
			&\le C (K_1 \tau^2+K_2 (h^{2q+2}+h^{2m+2}))+ \alpha_1 \tau \sum_{n=1}^{N-1}\norm{ e^{n+1}_h}_{\Gamma_h}^2.
		\end{align}
	
	An application of the discrete Gronwall's lemma to \eqref{error_2_2} with $T = N\tau$ yields the bound:
	\begin{equation}\label{eq:gronwall_result}
		\norm{e^N_h}_{\Gamma_h}^2 + \frac{ M\tau }{\epsilon^2} \sum_{n=1}^{N-1}\norm{ \zeta^{n+1}_h}^2_{\Gamma_h} \le c (1 + T\alpha_1 \exp(T\alpha_1)) \left(K_1 \tau^2 + K_2 (h^{2q+2} + h^{2m+2})\right),
	\end{equation}
	where the constants $K_1$ and $K_2$ are as defined in (\ref{gron_coff}).
	Next, we decompose the error into projection and discrete components:
	\[
	c(t^n) - c_h^n = \underbrace{c(t^n) - \pi_h c(t^n)}_{\eta^n} + \underbrace{\pi_h c(t^n) - c_h^n}_{e_h^n}.
	\]
	The triangle inequality  yields 
	\begin{equation}\label{eq:error_decomposition}
		\norm{c^e(t^n) - c_h^n}_{\Gamma_h} \le  \norm{\eta^n}_{\Gamma_h} + \norm{e_h^n}_{\Gamma_h}.
	\end{equation}
	
	For the projection error, we use the  approximation estimate on surfaces (\ref{proj_est}),
	\begin{align}\label{eq:projection_approx}
		\norm{\eta^n}_{\Gamma_h} = \norm{c^e(t^n) - \pi_h c(t^n)}_{\Gamma_h} \le C (h^{m+1}+h^{q+1}) \norm{c(t^n)}_{H^{m+1}(\Gamma)}.
	\end{align}
	
	Combining the estimate for $e_h^N$ from \eqref{eq:gronwall_result}  and \eqref{eq:projection_approx} gives:
	\begin{align*}
		\norm{c(t^N) &- c_h^N}_{\Gamma_h}^2 \le 2\left( \norm{\eta^N}_{\Gamma_h}^2 + \norm{e^N_h}_{\Gamma_h}^2 \right) \\
		&\le 2C\left( K_2 h^{2(m+1)} + C (1 + T\alpha_1 \exp(T\alpha_1)) \left(K_1 \tau^2 + K_2 (h^{2q+2} + h^{2m+2})\right) \right).
	\end{align*}
	Noting that $h^{2q+2} + h^{2m+2} \sim h^{2\min\{q, m\}+1}$, we consolidate terms to arrive at the final error estimate for the concentration:
	\begin{equation}\label{eq:final_error_concentration}
		\norm{c^e(t^N) - c_h^N}_{\Gamma_h}^2 \le C (1 + T\alpha_1 \exp(T\alpha_1)) \left( K_1 \tau^2 + K_2 h^{2\min\{m, q\}+1} \right).
	\end{equation}
	where the constants $K_1$ and $K_2$ are as defined in (\ref{gron_coff}).
	For the chemical potential,  the  error bound follows from \eqref{eq:gronwall_result} and the approximation result by the same argument.
	
\end{proof}

\section{Numerical results} \label{sec:numerics}
All experiments have been implemented using the open-source finite element software  Netgen/NGSolve \cite{JS14}.  The stabilization, mobility, and density parameters are chosen as  $\beta_s=1$, $M=1$, and $\rho=1$ in all the numerical experiments.

\subsection{Convergence Study}
To demonstrate the convergence behavior of the proposed method, we consider the following exact solution to the nonhomogeneous surface Cahn--Hilliard equations with the free energy per unit surface given by (\ref{eq_5}):
\begin{align}\label{eq95677}
	c^{*} = 0.5(1 - 0.8\exp(-0.4t))(x_1 x_2 + 1).
\end{align}
Here, $x = (x_1, x_2, x_3)^T$ denotes a point in $\mathbb{R}^3$. The exact chemical potential $\mu^*$ is computed directly from Equation (\ref{eq_3}), and we set $\epsilon = 0.1$.

The surface $\Gamma$ is characterized as the zero level set of the function 
\begin{align}\label{phi_zero}
	\phi(x) = \sqrt{x^2_1 + x^2_2 + x^2_3} - 1,
\end{align} 
embedded in the cubic domain $\Omega = [-5/4, 5/4]^3$. The initial triangulation $\mathcal{T}_h$ of $\Omega$ consists of 64 subcubes, each subdivided into six tetrahedra. The mesh is refined toward the surface, with refinement level $l$ and associated mesh size $h_l = 0.4/2^l$ for $l = 0, 1, 2, 3$. The computational mesh $\mathcal{T}^{\Gamma}_{h_l}$ comprises tetrahedra cut by the surface.

The discrete surface $\Gamma_h$ is defined as the zero level set of a piecewise polynomial interpolation $\phi_h$ of the exact level set function $\phi$, with geometry approximation order $q$ matching the finite element order ($q = 1$ for $\mathcal{P}_1$ and $q = 2$ for $\mathcal{P}_2$ elements). For $\mathcal{P}_2$ elements, an isoparametric mapping is employed to ensure higher-order geometry representation. All integrals over $\Gamma_h$ and in the narrow band $\Omega_h^\Gamma$ are computed using Gaussian quadrature rules of sufficient order to ensure that numerical integration errors do not dominate the discretization errors.
	
	The time step is chosen as $\tau = c_{\tau} \cdot h^{k+1}$ with $c_{\tau} = 0.5$, where $k$ is the finite element order. This scaling ensures that temporal discretization errors are of order $\mathcal{O}(h^{k+1})$ or smaller, which does not degrade the optimal spatial convergence rates of the method.

We report the following error norms in Figure \ref{num_sol_error}:
\begin{align}\label{norm_2}
	\|c^* - c_h\|_{2} = \left(\int_{\Gamma_h} |c^*(\mathbf{x},T) - c_h(\mathbf{x},T)|^2  ds\right)^{1/2},
\end{align}
\begin{align}\label{norm_1}
	\norm{c^*-c_h}_{2,2} = \left(\frac{1}{T} \sum_{k} \tau \norm{c^*(t_k)-c_h(t_k)}_{L^2(\Gamma_h)}^2\right)^{\frac{1}{2}}.
\end{align}

Figure \ref{num_sol_error} shows convergence plots for the scheme (\ref{dis_ch}) with $T = 0.1$, measuring errors in the norms defined in (\ref{norm_2}) and (\ref{norm_1}). The x-axis represents spatial step sizes, while the y-axis shows $L^2$-errors for both $c$ and $\mu$. Solid lines indicate $\mathcal{P}_1$ approximations, and dash-dotted lines represent $\mathcal{P}_2$ approximations. Dashed and dotted lines indicate the convergence slopes.


We observe $\mathcal{O}(h^{2})$ convergence in both norms for $c$ and $\mu$ with $\mathcal{P}_1$ elements, and $\mathcal{O}(h^{3})$ convergence with $\mathcal{P}_2$ elements. These computational results confirm optimal spatial convergence and are in agreement with theoretical predictions.

\begin{figure}[!ht]
	\centerline{
		\begin{tabular}{cc} 
			\hspace{0cm}
			\resizebox*{5.8cm}{!}{\includegraphics{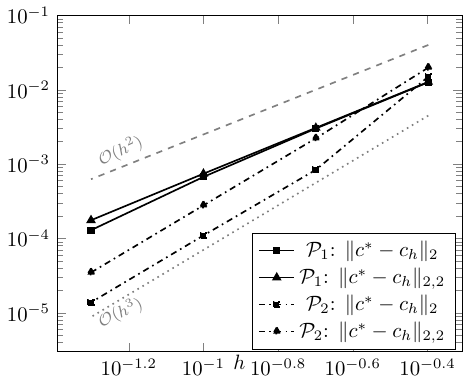}}%
			&\hspace{0cm}
			\resizebox*{5.8cm}{!}{\includegraphics{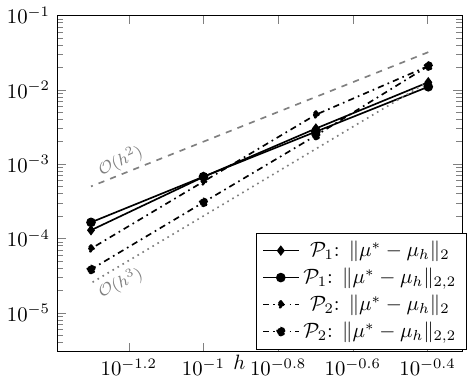}}%
			\\
			{(a)} \hspace{-1cm}&{(b) \  }
		\end{tabular}
	} 
	\caption{{{ Figures (a)--(b) present the convergence behavior the numerical solutions of CH equations  with the  the exact solution defined in  \eqref{eq95677}.}}}\label{num_sol_error}
\end{figure}

\subsection{Phase separation on a sphere}
The surface of a sphere provides a canonical setup for studying phase separation phenomena. Its geometric simplicity offers a tractable benchmark case, while its practical relevance is demonstrated by applications such as spherical lipid vesicles in drug delivery systems~\cite{yushutin2019computational}. We therefore proceed with the numerical investigation using this geometry.

The surface $\Gamma$ is defined as the zero level set of the scalar function $\phi$ in (\ref{phi_zero}).
The computational domain is embedded within $\Omega = [-5/4, 5/4]^3$, with spatial discretization parameter $h = 0.05$ ensuring sufficient resolution of the interfacial width. The Cahn--Hilliard interface parameter is set to $\epsilon = 0.05$, and adaptive time stepping with parameters $\tau$ specified in Table \ref{stoke_table1_31} maintains numerical stability while capturing the relevant dynamics. The system is initialized with a uniformly random concentration field $c_0(\mathbf{x}) = \text{rand}(\mathbf{x})$, where $\text{rand}(\mathbf{x}) \in [-1,1]$ represents small perturbations from the well-mixed state.

Figure \ref{num_sol_1as_NEW} illustrates the  evolution of the phase field over $t \in (0,1000]$. During the initial stage, the phase separation pattern develops rapidly, with distinct structures emerging by $T = 0.1$. This fast evolution phase concludes around $T = 5$, after which the dynamics slow down significantly. Beyond $T = 50$, the interfacial energy dissipation progresses at an even slower rate, indicating a gradual approach toward thermodynamic equilibrium.

The final state at $T = 1000$ (Figure~\ref{num_sol_1as_NEW}(l)) exhibits the expected equilibrium for a sphere.
In the sharp-interface limit ($\epsilon \to 0$), the Cahn--Hilliard dynamics approach a geometric evolution law where the interface evolves to minimize its length (perimeter) while conserving the enclosed area of each phase~\cite{cahn1996cahn, o2016cahn}. In our example with a sphere and 50\% of the area in each phase, this corresponds to the equilibrium state where each phase occupies a semisphere.

\begin{table}[h!]
	\begin{center}
		\caption{{{ Time steps used for the different time
					intervals to obtain the results in Figures \ref{num_sol_1as_NEW}--\ref{num_sol_new} }} }
		\label{stoke_table1_31}
		\begin{tabular}{|l|c|c|c|c|c|r|} 
			\hline
			Interval & $ [0,0.5)$ & $[0.5, 5)$& $ [5, 500)$ &$ [500, 5000)$ & $ [5000, 10000]$\\ [.75ex]
			\hline  
			$\tau$    & 0.01  & 0.1  &    1.0    &          10   &        100  \\ [.75ex]
			\hline
		\end{tabular}
	\end{center}
\end{table}

\begin{figure}[ht!]
	\centerline{\includegraphics[width=3.0cm]{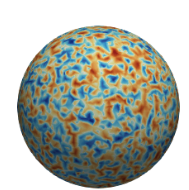}\hspace*{-0.1cm}\includegraphics[width=3.1cm]{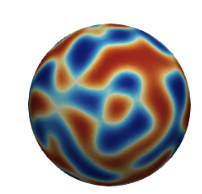}
		\hspace*{-0.2cm}\includegraphics[width=3.0cm]{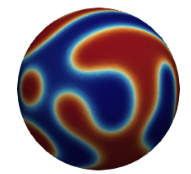}\hspace*{-0.1cm}\includegraphics[width=3.0cm]{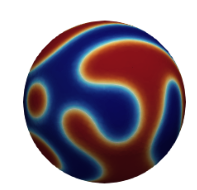}
	}
	\hspace{1.5cm} {(a) T=0 } \hspace{1.5cm}{(b) T=0.1 } \hspace{1.5cm}{(c) T=0.5 } \hspace{1.0cm}{(d) T=1 }
	
	\centerline{\includegraphics[width=3.1cm]{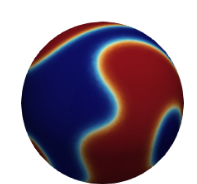}\hspace*{-0.1cm}\includegraphics[width=3.0cm]{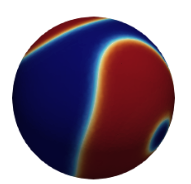}
		\hspace*{-0.2cm}\includegraphics[width=3.0cm]{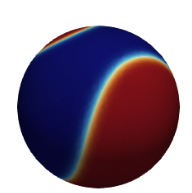}\hspace*{-0.1cm}\includegraphics[width=3.0cm]{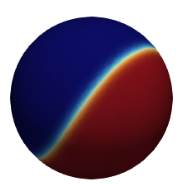}
	}
	\hspace{1.5cm} {(e) T=5 } \hspace{1.5cm}{(f) T=10 } \hspace{1.5cm}{(g) T=25 } \hspace{1.0cm}{(h) T=50 }
	
	\centerline{\includegraphics[width=3.0cm]{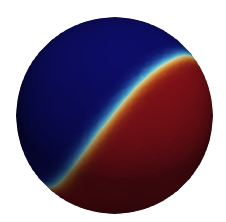}\hspace*{-0.1cm}\includegraphics[width=3.0cm]{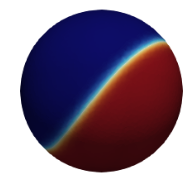}
		\hspace*{-0.1cm}\includegraphics[width=3.0cm]{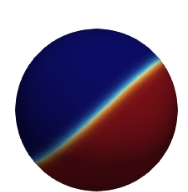}\hspace*{-0.1cm}\includegraphics[width=3.2cm]{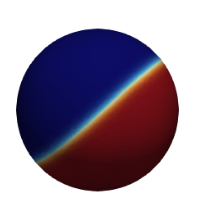}
	}
	\hspace{1.5cm} {(i) T=75 } \hspace{1.0cm}{(j) T=100 } \hspace{1.0cm}{(k) T=500 } \hspace{1.0cm}{(l) T=1000 }
	\caption{{Figures (a)--(l) present the phase separation with $\mathcal{P}_1$ approximation and $\epsilon=0.05$ and $\tau$ is given as in Table \ref{stoke_table1_31}.} }\label{num_sol_1as_NEW}
\end{figure}

\subsection{ Phase separation in a complex surface}   In this section, 
we consider the Cahn--Hilliard equations posed on a more complex manifold.
The surface $\Gamma$ is implicitly defined as the zero level-set of 
\begin{align*}
	\phi(x_1, x_2, x_3)= & (x^2_1 + x^2_2 - 4)^2 + (x^2_2 - 1)^2 +(x^2_2 + x^2_3 - 4)^2 + (x^2_1 - 1)^2 \\&+ (x^2_1 + x^2_3 - 4)^2 + (x^2_3 - 1)^2-13.
\end{align*}
The example of $\Gamma$ is taken from \cite{chernyshenko2015adaptive}. 
We embed $\Gamma$ in the cube  $[-9/4, 9/4]^3$ centered at the origin. We set $h=0.05$ and time step $\tau$ is given as in Table \ref{stoke_table1_31}. We set  $\epsilon= 0.05$. 
The system is initialized with a uniformly random concentration field $c_0(\mathbf{x}) = \text{rand}(\mathbf{x})$, where $\text{rand}(\mathbf{x}) \in [-1,1]$ represents small perturbations from the well-mixed state.
Figure \ref{num_sol_new} shows the computed concentration  up to $T=10^4$.
On the six–hole surface we observe a slower relaxation, and the apparent steady state lacks obvious symmetries. Two observations may explain this: (i) on higher-genus manifolds, area-constrained perimeter minimizers need not be unique and may break ambient symmetries \cite{ Canete2006Torus,CaneteRitore2008, HowardsHutchingsMorgan1999}; (ii) the surface Cahn--Hilliard dynamics is a nonconvex $H^{-1}$ gradient flow whose sharp-interface limit is a surface-diffusion–type motion (normal velocity involves surface second derivatives of curvature), so trajectories can settle near local rather than global minima and display slow coarsening \cite{BatesXun1994,CahnElliottNovickCohen1996,LeeMunchSuli2016}.


\section*{Acknowledgment} The authors were supported in part by the U.S. National Science Foundation under award
DMS-2408978.

\begin{figure}[ht!]
	\centerline{\includegraphics[width=3.5cm]{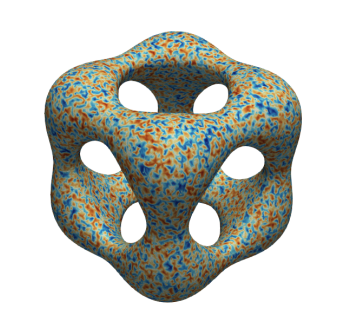}\hspace*{-0.1cm}\includegraphics[width=3.0cm]{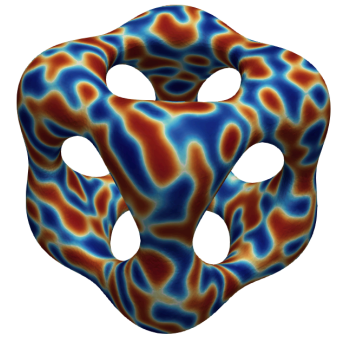}
		\hspace*{-0.2cm}\includegraphics[width=3.0cm]{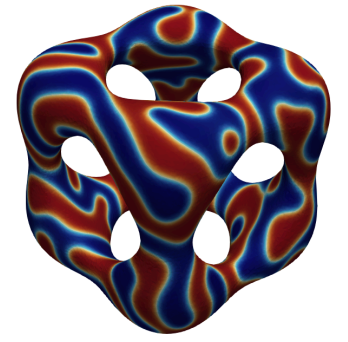}\hspace*{-0.1cm}\includegraphics[width=3.0cm]{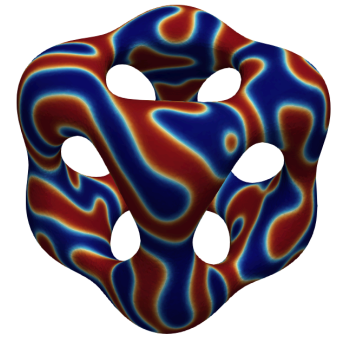}
	}
	\hspace{1.5cm} {(a) T=0 } \hspace{1.5cm}{(b) T=0.1 } \hspace{1.5cm}{(c) T=0.5 } \hspace{1.0cm}{(d) T=1 }
	
	\centerline{\includegraphics[width=3.0cm]{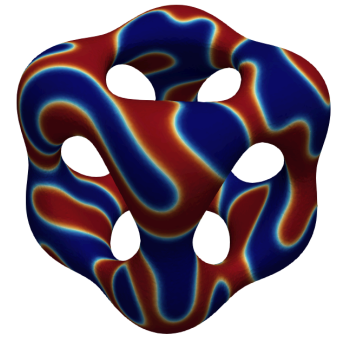}\hspace*{-0.1cm}\includegraphics[width=3.0cm]{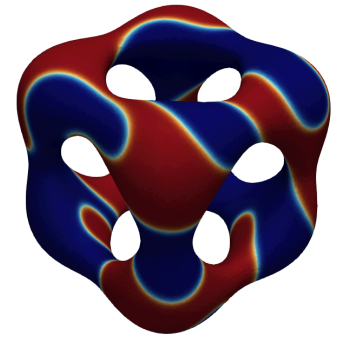}
		\hspace*{-0.2cm}\includegraphics[width=3.0cm]{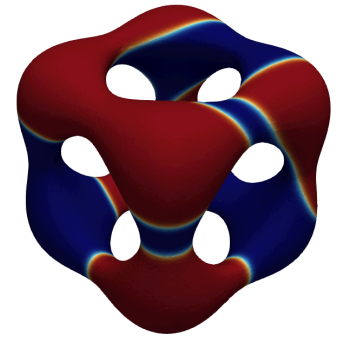}\hspace*{-0.1cm}\includegraphics[width=3.0cm]{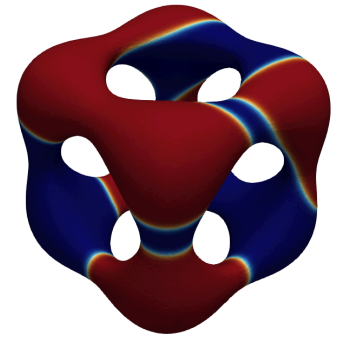}
	}
	\hspace{1.5cm} {(e) T=5 } \hspace{1.5cm}{(f) T=25 } \hspace{1.5cm}{(g) T=500 } \hspace{1.0cm}{(h) T=1000 }
	
	\centerline{\includegraphics[width=3.0cm]{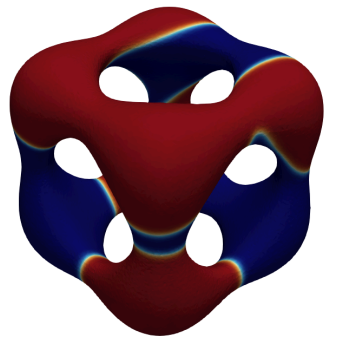}\hspace*{-0.1cm}\includegraphics[width=3.0cm]{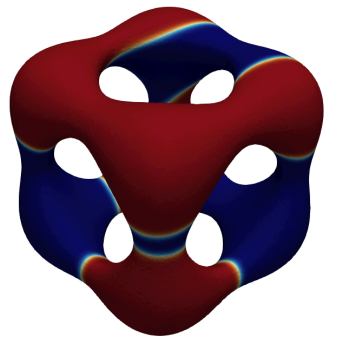}
		\hspace*{-0.1cm}\includegraphics[width=3.0cm]{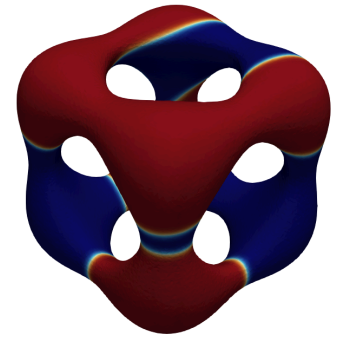}\hspace*{-0.0cm}\includegraphics[width=3.4cm]{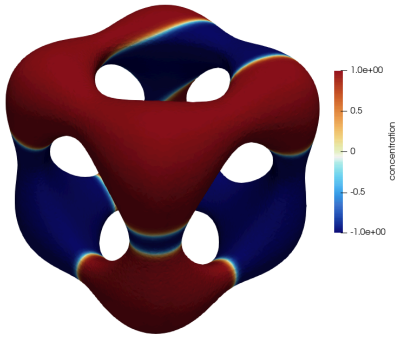}
	}
	\hspace{1.5cm} {(i) T=3000 } \hspace{1.0cm}{(j) T=5000 } \hspace{1.0cm}{(k) T=8000 } \hspace{1.0cm}{(l) T=10000 }
	\caption{{Figures (a)--(l) present the phase separation with $\mathcal{P}_1$ approximation at $\epsilon=0.05$ with $\tau$ is given as in Table \ref{stoke_table1_31}.} }\label{num_sol_new}
\end{figure}





\newpage
\bibliographystyle{siam}
\bibliography{references}{}

\begin{thebibliography}{10}

\bibitem{barrett2017finite}
{\sc John~W Barrett, Harald Garcke, and Robert N{\"u}rnberg}, {\em Finite
  element approximation for the dynamics of fluidic two-phase biomembranes},
  ESAIM: Mathematical Modelling and Numerical Analysis, 51 (2017),
  pp.~2319--2366.

\bibitem{BatesXun1994}
{\sc Peter~W. Bates and Jianping Xun}, {\em Metastable patterns for the
  {C}ahn--{H}illiard equation, part i}, Journal of Differential Equations, 111
  (1994), pp.~421--457.

\bibitem{cahn1996cahn}
{\sc John Cahn, Charles Elliott, and Amy Novick-Cohen}, {\em The
  {C}ahn--{H}illiard equation with a concentration dependent mobility: motion
  by minus the laplacian of the mean curvature}, European journal of applied
  mathematics, 7 (1996), pp.~287--301.

\bibitem{cahn1961spinodal}
{\sc John~W Cahn}, {\em On spinodal decomposition}, Acta metallurgica, 9
  (1961), pp.~795--801.

\bibitem{CahnElliottNovickCohen1996}
{\sc J.~W. Cahn, C.~M. Elliott, and A.~Novick-Cohen}, {\em The
  {C}ahn--{H}illiard equation with a concentration dependent mobility: motion
  by minus the laplacian of the mean curvature}, European Journal of Applied
  Mathematics, 7 (1996), pp.~287--301.

\bibitem{cahn1958free}
{\sc John~W Cahn and John~E Hilliard}, {\em Free energy of a nonuniform system.
  i. interfacial free energy}, The Journal of chemical physics, 28 (1958),
  pp.~258--267.

\bibitem{Canete2006Torus}
{\sc Antonio Cañete}, {\em Stable and isoperimetric regions in rotationally
  symmetric tori with decreasing gauss curvature}, Indiana University
  Mathematics Journal, 55 (2006), pp.~1629--1659.

\bibitem{CaneteRitore2008}
{\sc Antonio Cañete and Manuel Ritor{\'e}}, {\em The isoperimetric problem in
  complete annuli of revolution with increasing gauss curvature}, Proceedings
  of the Royal Society of Edinburgh, Section A, 138 (2008), pp.~989--1003.

\bibitem{chernyshenko2015adaptive}
{\sc Alexey Chernyshenko and Maxim Olshanskii}, {\em An adaptive octree finite
  element method for pdes posed on surfaces}, Computer Methods in Applied
  Mechanics and Engineering, 291 (2015), pp.~146--172.

\bibitem{du2011finite}
{\sc Qiang Du, Lili Ju, and Li~Tian}, {\em Finite element approximation of the
  {C}ahn--{H}illiard equation on surfaces}, Computer Methods in Applied
  Mechanics and Engineering, 200 (2011), pp.~2458--2470.

\bibitem{elliott1989second}
{\sc Charles Elliott, Donald French, and F~Milner}, {\em A second order
  splitting method for the cahn-hilliard equation}, Numerische Mathematik, 54
  (1989), pp.~575--590.

\bibitem{elliott2025fully}
{\sc Charles Elliott and Thomas Sales}, {\em A fully discrete evolving surface
  finite element method for the {C}ahn--{H}illiard equation with a regular
  potential}, Numerische Mathematik, 157 (2025), pp.~663--715.

\bibitem{garcke2016coupled}
{\sc Harald Garcke, Johannes Kampmann, Andreas R{\"a}tz, and Matthias
  R{\"o}ger}, {\em A coupled surface-{C}ahn--{H}illiard bulk-diffusion system
  modeling lipid raft formation in cell membranes}, Mathematical Models and
  Methods in Applied Sciences, 26 (2016), pp.~1149--1189.

\bibitem{garcke2022viscoelastic}
{\sc Harald Garcke, Bal{\'a}zs Kov{\'a}cs, and Dennis Trautwein}, {\em
  Viscoelastic {C}ahn--{H}illiard models for tumor growth}, Mathematical Models
  and Methods in Applied Sciences, 32 (2022), pp.~2673--2758.

\bibitem{garcke2016cahn}
{\sc Harald Garcke, Kei~Fong Lam, Emanuel Sitka, and Vanessa Styles}, {\em A
  {C}ahn--{H}illiard--darcy model for tumour growth with chemotaxis and active
  transport}, Mathematical Models and Methods in Applied Sciences, 26 (2016),
  pp.~1095--1148.

\bibitem{gera2017cahn}
{\sc Prerna Gera and David Salac}, {\em {C}ahn--{H}illiard on surfaces: A
  numerical study}, Applied Mathematics Letters, 73 (2017), pp.~56--61.

\bibitem{GLR:2018}
{\sc J\"org Grande, Christoph Lehrenfeld, and Arnold Reusken}, {\em Analysis of
  a high-order trace finite element method for {PDE}s on level set surfaces},
  SIAM J. Numer. Anal., 56 (2018), pp.~228--255.

\bibitem{greer2006fourth}
{\sc John Greer, Andrea Bertozzi, and Guillermo Sapiro}, {\em Fourth order
  partial differential equations on general geometries}, Journal of
  Computational Physics, 216 (2006), pp.~216--246.

\bibitem{HowardsHutchingsMorgan1999}
{\sc Hugh Howards, Michael Hutchings, and Frank Morgan}, {\em The isoperimetric
  problem on surfaces}, American Mathematical Monthly, 106 (1999),
  pp.~430--439.

\bibitem{jeong2015microphase}
{\sc Darae Jeong and Junseok Kim}, {\em Microphase separation patterns in
  diblock copolymers on curved surfaces using a nonlocal cahn-hilliard
  equation}, The European Physical Journal E, 38 (2015), p.~117.

\bibitem{LeeMunchSuli2016}
{\sc Alpha~A. Lee, Andreas M{\"u}nch, and Endre S{\"u}li}, {\em Sharp-interface
  limits of the {C}ahn--{H}illiard equation with degenerate mobility}, SIAM
  Journal on Applied Mathematics, 76 (2016), pp.~433--456.

\bibitem{LOX:2018}
{\sc Christoph Lehrenfeld, Maxim Olshanskii, and Xianmin Xu}, {\em A stabilized
  trace finite element method for partial differential equations on evolving
  surfaces}, SIAM J. Numer. Anal., 56 (2018), pp.~1643--1672.

\bibitem{li2017unconditionally}
{\sc Yibao Li, Junseok Kim, and Nan Wang}, {\em An unconditionally
  energy-stable second-order time-accurate scheme for the {C}ahn--{H}illiard
  equation on surfaces}, Communications in Nonlinear Science and Numerical
  Simulation, 53 (2017), pp.~213--227.

\bibitem{miller1995spinodal}
{\sc MK~Miller, JM~Hyde, MG~Hetherington, A~Cerezo, GDW Smith, and CM~Elliott},
  {\em Spinodal decomposition in fe-cr alloys: Experimental study at the atomic
  level and comparison with computer models—i. introduction and methodology},
  Acta Metallurgica et Materialia, 43 (1995), pp.~3385--3401.

\bibitem{nitschke2012finite}
{\sc Ingo Nitschke, Axel Voigt, and J{\"o}rg Wensch}, {\em A finite element
  approach to incompressible two-phase flow on manifolds}, Journal of Fluid
  Mechanics, 708 (2012), pp.~418--438.

\bibitem{o2016cahn}
{\sc David O'Connor and Bjorn Stinner}, {\em The cahn-hilliard equation on an
  evolving surface}, arXiv preprint arXiv:1607.05627,  (2016).

\bibitem{membr1}
{\sc Maxim Olshanskii and Annalisa Quaini}, {\em Phase-separated lipid
  vesicles: Continuum modeling, simulation, and validation}, Physics of Fluids,
  37 (2025), p.~071304.

\bibitem{olshanskii2018trace}
{\sc Maxim Olshanskii and Arnold Reusken}, {\em Trace finite element methods
  for pdes on surfaces}, in Geometrically Unfitted Finite Element Methods and
  Applications: Proceedings of the UCL Workshop 2016, Springer, 2018,
  pp.~211--258.

\bibitem{reusken2015analysis}
{\sc Arnold Reusken}, {\em Analysis of trace finite element methods for surface
  partial differential equations}, IMA Journal of Numerical Analysis, 35
  (2015), pp.~1568--1590.

\bibitem{JS14}
{\sc Joachim Sch\"{o}berl}, {\em C++11 implementation of finite elements in
  {NGS}olve}, tech. report, ASC-2014-30, Institute for Analysis and Scientific
  Computing, September 2014.

\bibitem{shen2010numerical}
{\sc Jie Shen and Xiaofeng Yang}, {\em Numerical approximations of allen-cahn
  and {C}ahn--{H}illiard equations}, Discrete Contin. Dyn. Syst, 28 (2010),
  pp.~1669--1691.

\bibitem{wang2023fusogenicity}
{\sc Yifei Wang, Yerbol Palzhanov, Dang~T Dang, Annalisa Quaini, Maxim
  Olshanskii, and Sheereen Majd}, {\em On fusogenicity of positively charged
  phased-separated lipid vesicles: experiments and computational simulations},
  Biomolecules, 13 (2023), p.~1473.

\bibitem{yushutin2019computational}
{\sc Vladimir Yushutin, Annalisa Quaini, Sheereen Majd, and Maxim Olshanskii},
  {\em A computational study of lateral phase separation in biological
  membranes}, International journal for numerical methods in biomedical
  engineering, 35 (2019), p.~e3181.

\bibitem{zhiliakov2021experimental}
{\sc Alexander Zhiliakov, Yifei Wang, Annalisa Quaini, Maxim Olshanskii, and
  Sheereen Majd}, {\em Experimental validation of a phase-field model to
  predict coarsening dynamics of lipid domains in multicomponent membranes},
  Biochimica et Biophysica Acta (BBA)-Biomembranes, 1863 (2021), p.~183446.

\end{thebibliography}
\end{document}